\documentclass[10pt]{article}

\usepackage{dsfont}
\usepackage{cuted}
\usepackage{bm}
\usepackage{color}
\usepackage{psfrag}

\usepackage{amssymb,amsmath,amsfonts,amsthm,mathrsfs}
\usepackage{bbm}
\usepackage{color}

\usepackage[symbol]{footmisc}

\usepackage{graphicx}
\usepackage{caption}

\usepackage[margin = 1em]{subfig}
\usepackage{geometry}
\usepackage{pgfplots}
\pgfplotsset{compat = 1.11}
\usetikzlibrary{calc,decorations.pathmorphing,decorations.pathreplacing,patterns}
\usepgfplotslibrary{colormaps} 
\usetikzlibrary{arrows.meta}
\tikzset{>={Latex[width=1.7mm,length=2.2mm]}}
\usetikzlibrary{decorations.text}

\newtheorem{defn}{\noindent \bf{Definition}}[section]
\newtheorem{thm} {\noindent \bf{Theorem}}[section]

\newtheorem{remark}{\noindent \bf{Remark}}[section]

\newtheorem{prop}{\noindent \bf{Proposition}}[section]
\newtheorem{cor}{\noindent \bf{Corollary}}[section]
\newtheorem{condition}{\noindent \bf{Conditions}}[section]
\newtheorem{asmp}{\noindent \bf{Assumptions}}[section]

\begin{document}

\title{{\bf Submission for the MAM9 special issue:} \\
Matrix geometric approach for random walks: stability condition and equilibrium distribution}
\author{
        Stella Kapodistria\footnotemark[1] \and
Zbigniew Palmowski  \footnotemark[2] }

\date{\today}

\maketitle

\footnotetext[1]{Department of Mathematics and Computer Science, Eindhoven University of Technology,
P.O.\ Box 513, 5600 MB  Eindhoven, The Netherlands,
E-mail:   \texttt{s.kapodistria@tue.nl}}
\footnotetext[2]{Faculty of Pure and Applied Mathematics, Wroclaw University of Science and Technology, Wyb. Wyspia\'nskiego 27, 50-370 Wroclaw, Poland,
E-mail:   \texttt{zbigniew.palmowski@gmail.com}}

\begin{abstract}
In this paper, we analyse a sub-class of two-dimensional homogeneous nearest neighbour (simple) random walk restricted on the lattice using the matrix geometric approach. In particular, we first present an alternative approach for the calculation of the stability condition, extending the result of Neuts drift conditions \cite{Neuts} and connecting it with the result of Fayolle et al. which is based on Lyapunov functions \cite{FIM}. Furthermore, we consider the sub-class of random walks with equilibrium distributions given as series of product-forms and, for this class of random walks, we calculate the eigenvalues and the corresponding eigenvectors of the infinite matrix $\bm{R}$ appearing in the matrix geometric approach. This result is obtained by connecting and extending three existing approaches available for such an analysis: the matrix geometric approach, the compensation approach and the boundary value problem method. In this paper, we also present the spectral properties of the infinite matrix $\bm{R}$.\\

\noindent
{\bf Keywords}: Random walks; Stability condition; Equilibrium distribution; Matrix geometric approach; Spectrum; Compensation approach; Boundary value problem method.
\end{abstract}

\section{Introduction}
The objective of this work is to demonstrate how to obtain  the stability condition and the equilibrium distribution of the state of a two-dimensional homogeneous nearest neighbour (simple) random walk restricted on the lattice using it's underlying Quasi-Birth-Death (QBD) structure and the matrix geometric approach. This type of random walk can be modelled as a QBD process with the  characteristic that both the levels and the phases are countably infinite. Then, based on the matrix geometric approach, if $\bm{\pi}_{n}=\begin{array}{c c c }( \pi_{n,0} & \pi_{n,1}& \cdots)\end{array}$ denotes the  vector of the equilibrium distribution at level $n$, $n=0,1,\ldots$, it is known that $\bm{\pi}_{n+1}=\bm{\pi}_{n}\bm{R}$, $n\geq 1$. This is a very well known result, but the complexity of the solution lies in the calculation of the infinite dimension matrix $\bm{R}$. In this paper, we investigate how the matrix geometric approach can be extended to the case of countably infinite phases and we discuss the challenges that arise with such an extension. Furthermore, we propose an approach for the calculation of the eigenvalues and eigenvectors of matrix $\bm{R}$, that complements the existing approaches for this type of random walks. Moreover, this approach can be numerically used for the approximation of the matrix $\bm{R}$ by considering spectral truncation instead of the usual state space truncation.\\

\subsection{Literature overview}
The main body of literature on the topic of countably infinite phases of QBDs is devoted to either stochastic processes in which the matrix $\bm{R}$ has a simple structure or to asymptotic results concerning the decay rate. \\

Regarding QBDs with a simple structure the literature is mainly devoted to stochastic processes for which the rate matrix has a simple property, see, e.g., \cite{Due} and the references therein,  or to stochastic process with a special structure in the allowed transitions, see, e.g., \cite{Katehakis2,Katehakis1} and the references therein. In \cite{Due}, the authors consider an infinite rate matrix $\bm{R}$ that can be written as the product of a vector column times a row vector. This simple structure permits, under the assumptions of homogeneity and irreducibility, to show that the stationary distributions of these processes have a product form structure as a function of the level. Additionally, they apply their results in the case of the Cox(k)/M${}^Y$/1 queue. In \cite{Katehakis2,Katehakis1} and the references therein, the authors investigate stochastic processes with a special structure regarding the allowed transitions satisfying the so called successive lumpability property. This class of Markov chains is specified through its property of calculating the stationary probabilities of the system by successively computing the stationary probabilities of a propitiously constructed sequence of Markov chains. Each of the latter chains has a (typically much) smaller state space and this yields significant computational improvements. Such models are a special class of the GI/M/1 or the M/G/1 type of QBD processes. In \cite{Katehakis2,Katehakis1}, the authors discuss how the results for discrete time Markov chains extend to semi-Markov processes and continuous time Markov processes. Furthermore, in \cite{Katehakis1}, the authors compare the successive lumping methodology developed with the  lattice path counting approach \cite{24} for the calculation of the rate matrix of a queueing model. These two methodologies are compared both in terms of applicability requirements and numerical complexity by analysing their performance for some classical queueing models. Their main findings are: i) When both methods are applicable the successive lumping based algorithms outperform the lattice path counting algorithm. ii) The successive lumping algorithms, contrary to the lattice path counting algorithm, includes a method to compute the steady state distribution using this rate matrix.\\

Regarding the work devoted on the decay rate of QBDs with infinite phases, in \cite{Takahashi}, the authors provide a sufficient condition for the geometric decay of the steady-state probabilities of a QBD process with a countable number of phases in each level. As an example, the authors apply the result to a two-queue system with the shorter queue discipline. In addition, in \cite{Haque}, the authors present sufficient conditions, under which the stationary probability vector of a QBD process with both infinite levels and phases decays geometrically. They also present a method to compute the convergence norm and the left-invariant vector of the rate matrix $\bm{R}$ based on spectral properties of the censored matrix of a matrix function constructed with the repeating blocks of the transition matrix of the QBD process. Their method reduces the determination of the convergence norm to the determination of the zeros of a polynomial. In \cite{Motyer}, the authors consider the class of level-independent QBD processes that have countably many phases and generator matrices with tridiagonal blocks that are themselves tridiagonal and phase independent. They derive simple conditions for possible decay rates of the stationary distribution of the `level' process. Their results generalise those of Kroese, Scheinhardt, and Taylor \cite{Taylor}, who studied in detail a particular example, the tandem Jackson network.\\

Special attention is also paid to the explicit calculation of the decay rates. In \cite{Li}, the authors investigate the geometric tail decay of the stationary distribution for a GI/G/1 type quasi-birth-and-death process and apply their result to a generalised join-the-shortest-queue model. Furthermore, they establish the geometric tail asymptotics along the direction of the difference between the two queues. In \cite{Miyazawa}, the author considers a double QBD process, i.e. a skip-free random walk in the two-dimensional positive quadrant with homogeneous reflecting transitions at each boundary face. For this stochastic process, the author investigates the tail decay behaviour of the stationary distribution using the matrix geometric method. Furthermore, the author exemplifies the decay rates for Jackson networks and for some modifications of Jackson networks.\\

The effect of state space truncation in the case of countably infinite phases has also received considerable attention. In \cite{Latouche1}, the authors study a system of two queues with boundary assistance, represented as a continuous-time QBD process. Under their formulation, the QBD at hand has a `doubly infinite' number of phases. For this process, the authors determine the convergence norm of the rate matrix $\bm{R}$ and, consequently, the interval in which the decay rate of the infinite system can lie. Moreover, they consider four sequences of finite-phase approximations: one is derived by truncating the infinite system without augmentation, the others are obtained by using different augmentation schemes that ensure that the generator of the QBD remains conservative. The sequences of rate matrices for the truncated system without augmentation and one of the sequences with augmentation have monotonically increasing spectral radii that approach the convergence norm of the infinite matrix $\bm{R}$ as the truncation point tends to infinity; the two other sequences of finite matrices have spectral radii that are constant irrespective of the truncation size, and not equal to the convergence norm of the infinite matrix.

\subsection{Background: finite phases}
Consider a QBD with finite phases, say $m<\infty$, and infinitesimal generator
\begin{eqnarray}\label{generator_matrix}
\bm{Q}=\begin{bmatrix}
  \bm{B} & \bm{A}_{1}&0 &0&\cdots \\
  \bm{A}_{-1}  & \bm{A}_{0} &\bm{A}_{1}  &0&\cdots \\
0&\bm{A}_{-1} & \bm{A}_{0} &\bm{A}_{1}  &\cdots \\
0&0&\bm{A}_{-1}  & \bm{A}_{0}  &\cdots \\
\vdots&\vdots&\vdots&\vdots&\ddots
 \end{bmatrix} ,
 \end{eqnarray}
where the matrices $\bm{A}_1$ and $\bm{A}_{-1}$ are nonnegative, and the matrices $\bm{A}_0$ and $\bm{B}$ have nonnegative off-diagonal elements and strictly negative diagonals. The row sums of $\bm{Q}$ are equal to zero and we assume that the process is irreducible.\\

The stability condition (sufficient and necessary condition) of such a QBD, see  \cite[Theorem ~1.7.1]{Neuts},  can be obtained by the drift condition
 \begin{eqnarray}\label{DriftCondition}
\bm{x}'\bm{A}_{1}\mathds{1}<\bm{x}'\bm{A}_{-1}\mathds{1},
 \end{eqnarray}
with $\bm{A}_{1}$, $\bm{A}_{0}$, $\bm{A}_{-1}$ the transition $m\times m$ rate matrices capturing the rates to a higher level, within the same level and to a lower level, respectively, and $\bm{x}'$ a row vector obtained as  the unique solution  to
$\bm{x}'(\bm{A}_{1}+\bm{A}_{0}+\bm{A}_{-1})=0$ such that $\bm{x}' \mathds{1}  =1$, with $\mathds{1}$ a column vector of ones. \\

The equilibrium distribution of the QBD, see \cite[Theorem~6.4.1]{Latouche},  can be obtained as
\begin{eqnarray*}\bm{\pi}_{n+1}=\bm{\pi}_{n}\bm{R}\ \Rightarrow\ \bm{\pi}_{n}=\bm{\pi}_{0}\bm{R}^{n},\ n=0,1,\ldots, \end{eqnarray*}
with $\bm{\pi}_{n}=\begin{array}{c c c c}( \pi_{n,0} & \pi_{n,1}& \cdots& \pi_{n,m-1}).\end{array}$
The matrix $\bm{R}$ records the rate of sojourn in the states of level $n+1$ per unit of the local time of level $n$.
Furthermore, $\bm{R}=\bm{A}_1\bm{N}$, where the matrix $\bm{N}$ records the expected sojourn times in the states of the level $n$, starting from level $n$, before the first visit to level $n-1$. \\

If the rate matrix $\bm{R}$ is decomposable (i.e., Jordan canonical form), then the equilibrium distribution can be written as \begin{equation}\bm{\pi}_n=\sum_{k} \alpha^n_k \bm{p}_k,\ n=0,1,\ldots,\label{background}
\end{equation}
where  $\alpha_k$ are the  eigenvalues of matrix $\bm{R}$ and $\bm{p}_k$ are the left eigenvectors. The above result is based on the analysis presented in \cite[Section~1.6]{Neuts} and the references therein. Most typically, further analysis is restricted to the  case the roots $\alpha_k$ are distinct. In such a case, the vectors $\bm{p}_k$ are then determined up to  a multiplicative constant. Such results can be proven in a straightforward way using spectral theory, see, e.g., \cite{Hunter, Reed}: For an $m\times m$ complex matrix  $\bm{R}$, with $m$ distinct eigenvalues $\{\alpha_1,\ldots,\alpha_m\}$,  there exists an orthonormal basis of $\mathbb{C}^{m\times m}$, say $\{\bm{E}_1, \ldots, \bm{E}_m\}$, with $\bm{E}_k\in\mathbb{C}^{m\times m}$, and $\bm{R}$ can be unitarily diagonalised as follows
\[{\displaystyle  \bm{R} = \sum_{k=1}^m \alpha_k\bm{E}_k.}\]
This diagonalisation yields that ${\bm{R}^n = \sum_{k=1}^m \alpha_k^n\bm{E}_k,\ n=0,1,\ldots}$. Then, it is evident that
\begin{align*}
\bm{\pi}_n&=\bm{\pi}_0\bm{R}^n= \sum_{k=1}^m \alpha_k^n \bm{\pi}_0\bm{E}_k=\sum_{k=1}^m \alpha^n_k \bm{p}_k,\ n=0,1,\ldots,
\end{align*}
see \cite[Section~9.1]{Hunter}.\\

In this paper, we describe how to extend the drift condition \eqref{DriftCondition} to the case of random walks on the lattice, and we investigate the extensions of spectral theory, and in particular of Equation \eqref{background}, in the case of  infinite ``diagonalisable'' matrices. In this case, we also investigate how to obtain the eigenvalues and the eigenvectors of the rate matrix. To this end, we restrict our analysis to random walks whose invariant measures can be viewed as a series extension of Equation \eqref{background}. \\

The paper is organised as follows: in Section \ref{sec-2} the model is described and in Section \ref{sec-3.1.1} the stability condition is derived. In Section \ref{sec-3} the three relevant methods for the calculation of the equilibrium distribution are sketched; more concretely,  the matrix geometric approach is presented in Section \ref{sec-3.1.2}, the compensation approach in Section \ref{sec-3.2} and the boundary value problem method in Section \ref{sec-3.3}. In Section \ref{sec-4}, we present structural and spectral properties of the invariant measure and the corresponding infinite dimension matrix $\bm{R}$.
In Section \ref{sec-6}, we show how to calculate recursively the eigenvalues of matrix $\bm{R}$, Section \ref{sec-6.1}, and the corresponding eigenvectors, Section \ref{sec-6.2}. Thereafter, in Section \ref{sec-6.3}, we describe the algorithm for the recursive calculation of the eigenvalues and eigenvectors, and in the subsequent section, \ref{sec-6.4}, we discuss the numerical properties of spectral truncation for the matrix $\bm{R}$. In Section \ref{sec-7}, we illustrate the application of our approach to the paradigm of the join the shortest queue and compare our results with the relevant literature. Finally, in Section \ref{Conclusions}, conclusions and future work is discussed.

\section{Model description }\label{sec-2}
As a first step and for illustration purposes, as well as for reasons of simplicity, we restrict our analysis to  a class of Markov processes on the lattice in the positive quadrant of $\mathbb{R}^2$, random walks, whose equilibrium distribution away from the origin $(0,0)$ can be written as a series (finite or infinite) of product-forms. More concretely, we consider random walks for which the transition rates are constant, i.e. they do not depend on the state, and we further assume that transitions are restricted to neighbouring states. However, one has to be careful, since not all random walks with constant transition rates to neighbouring states have an equilibrium distribution that can be written as series of product-forms. For this reason in this paper, we restrict our analysis to a sub-class of such random walks for which it is guaranteed that the equilibrium distribution will have the desired form.
The transition rates of the sub-class of interest to us are depicted in Figure \ref{fig:trd}. As it is evident from the figure, such processes can be modelled using the theory of QBDs.  This is due to the nearest neighbour transition structure.\\

\begin{figure}%
\centering%
\includegraphics{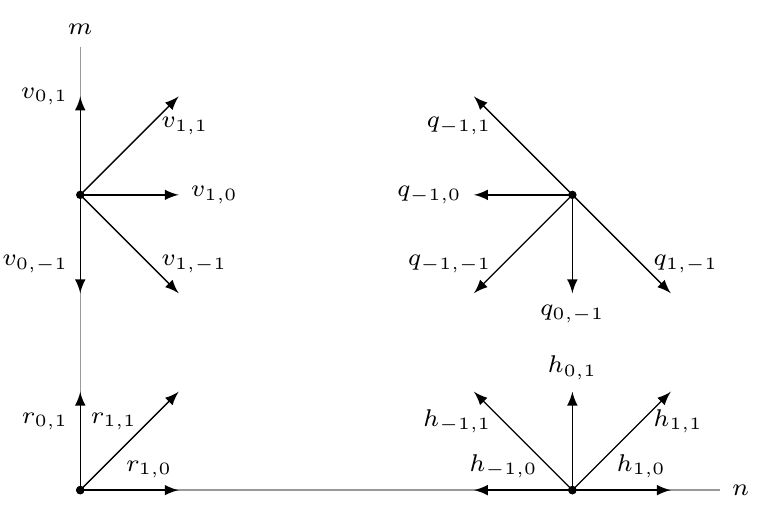}
\caption{Transition rate diagram of the homogeneous simple random walk on the state space $(n,m)$ with no transitions in the interior to the North, North-East, and East. Only the transitions at a few selected states are depicted as an indication.}%
\label{fig:trd}
\end{figure}%

Under the following necessary conditions, referred to as {\em  conditions for meromorphicity}, cf. \cite{MR1138205, ChenPhD,CohenJAP}, the invariant measure  of the simple random walk restricted on the lattice can be written as an infinite series of product-forms for all states away from the origin $(0,0)$.
\begin{condition}\label{conditions-CA}
\begin{description}	
\item[Step size:] Only transitions to neighbouring states are allowed;
\item[Forbidden steps:] No transitions from interior states to the North, North-East, and East are allowed;
\item[Homogeneity:] In the interior, all transitions in the same direction occur according to the same rate;
\end{description}
\end{condition}

Furthermore, in order to avoid the random walk exhibiting a trivial behaviour, we consider the following assumptions.
\begin{asmp}\label{Assumptions}
\begin{itemize}
\item[i)] Non-zero rate to the South: $q_{-1,-1}+q_{0,-1}+q_{1,-1}>0$;
\item[ii)] Non-zero rate to the West: $q_{-1,-1}+q_{-1,0}+q_{-1,1}>0$;
\item[iii)] Non-zero reflecting rate for the horizontal axis: $h_{-1,1}+h_{0,1}+h_{1,1}>0$;
\item[iv)] Non-zero reflecting rate for the vertical axis: $v_{1,1}+v_{1,0}+v_{1,-1}>0$;
\item[v)] Non-zero rate out of state 0: $r_{0,1}+r_{1,1}+r_{1,0}>0$.
\end{itemize}
\end{asmp}

For the random walk under consideration, the invariant measure has the desired structure, that of series of product-forms, as stated in the following theorem.
\begin{thm}\label{Thm2.33}{\normalfont \cite[Theorem 2.33]{MR1138205}}
Under the Conditions for meromorphicity \ref{conditions-CA} and Assumptions \ref{Assumptions}, and given the stability condition, there exists an $N\in\mathbb{Z}_+$, such that for $n+m>N$, the invariant measure $\pi_{n,m}$ can be written
\begin{align}\label{mainInvM}
\pi_{n,m}=\sum_{(\alpha_0,\beta_0)}c(\alpha_0,\beta_0) x_{n,m}(\alpha_0,\beta_0),
\end{align}
where $(\alpha_0,\beta_0)$ runs through the set of at most four feasible pairs and $c(\alpha_0,\beta_0)$ is an appropriately chosen coefficient and
\begin{align}
x_{n,m}(\alpha_0,\beta_0)&=c_0\alpha_0^n\beta_0^m+\sum_{k=1}^\infty c_k\alpha_k^n(\beta_{k-1}^m+f_k\beta_k^m)
,\ n,m>0,\label{1}\\
x_{n,0}(\alpha_0,\beta_0)&=\sum_{k=0}^\infty e_k\alpha_k^n,\ n>0,\label{pin0}\\
x_{0,m}(\alpha_0,\beta_0)&=\sum_{k=0}^\infty d_k\beta_k^m,\ m>0.\label{pi0m}
\end{align}
\end{thm}
The invariant measure of the states close to the origin, e.g. $\pi_{0,0}$, can be obtained as a function of \eqref{1}-\eqref{pi0m} by solving the corresponding system of equilibrium equations\.\\

In order to simplify the notation, in the sequel we assume that there is a single $(\alpha_0,\beta_0)$ pair (out of the possible four) and that $N=0$. The latter assumption can be easily justified using the QBD structure of the system and spectral theory. This is a topic that we will further discuss in the sequel, see Remark \ref{remark-As2}. For now, we state the above as assumptions, but the only purpose of these two assumptions is the simplification of the notation and the easier exposition of the results.
\begin{asmp}\label{Assumptions2}
\begin{itemize}
\item[vi)] The sum in Equation \eqref{mainInvM} runs through a single pair $(\alpha_0,\beta_0)$, and can be therefore simplified. Furthermore, in order to further simplify the notation we set $c(\alpha_0,\beta_0)=1$.
\item[vii)] Equation \eqref{mainInvM} is valid for all $n+m>0$, i.e. Equation \eqref{mainInvM} is not valid only for state $(0,0)$.
\end{itemize}
\end{asmp}

Some examples of queueing systems that fall in the sub-class of random walks depicted in Figure \ref{fig:trd} and satisfy Assumptions \ref{Assumptions}--\ref{Assumptions2}  are the $2\times2$ switch and the join the shortest queue, see e.g.,  \cite{MR1833660} and the references therein. Also, it is notable to mention that there exist random walks that violate the conditions for meromorphicity, but still exhibit an invariant measure which can be written as a finite sum of product-forms, such as e.g., the two-station Jackson networks. For an extensive treatment of random walks, with an invariant measure representable with a finite sum of product-forms, the interested reader is referred to \cite{Chen}. In this paper, we focus on the spectral properties of processes satisfying the conditions for meromorphicity, but our results can be extended in the case the invariant measure is represented by a finite sum of product-forms. Such spectral properties were first investigated in \cite{Taylor}.  \\

\subsection{QBD structure\label{sec-3.1}}
For the model described in the section above (cf. Figure \ref{fig:trd}), we can define either $n$ to be the {\em level} and $m$  the {\em phase} or the reverse, i.e. $m$ to be the  level and $n$  the phase. For the derivation of the stability condition this flexibility will prove valuable, while in the case of the calculation of the equilibrium distribution it is crucial which dimension plays the role of the level. To this purpose, we provide two generator representations, the first say $\bm{G}^{\footnotesize H}$, corresponds to the case where $n$ is chosen as the  level, while the second  $\bm{G}^{\footnotesize V}$, corresponds to the case where $m$ is chosen as the level. The generators (depending on the choice of the level) of the random walk can be written as follows
\begin{eqnarray}\label{generator_matrix}
\bm{G}^{\footnotesize H}=\begin{bmatrix}
  \bm{B}_0^{\footnotesize H} & \bm{B}_{1}^{\footnotesize H}&0 &0&\cdots \\
  \bm{A}_{-1}^{\footnotesize H} & \bm{A}_{0}^{\footnotesize H}&\bm{A}_{1}^{\footnotesize H} &0&\cdots \\
0&\bm{A}_{-1}^{\footnotesize H} & \bm{A}_{0}^{\footnotesize H}&\bm{A}_{1}^{\footnotesize H} &\cdots \\
0&0&\bm{A}_{-1}^{\footnotesize H} & \bm{A}_{0}^{\footnotesize H} &\cdots \\
\vdots&\vdots&\vdots&\vdots&\ddots
 \end{bmatrix} \ \text{and}\
\bm{G}^{\footnotesize V}=\begin{bmatrix}
  \bm{B}_0^{\footnotesize V} & \bm{B}_{1}^{\footnotesize V}&0 &0&\cdots \\
  \bm{A}_{-1}^{\footnotesize V} & \bm{A}_{0}^{\footnotesize V}&\bm{A}_{1}^{\footnotesize V}&0&\cdots \\
0&\bm{A}_{-1}^{\footnotesize V} & \bm{A}_{0}^{\footnotesize V}&\bm{A}_{1}^{\footnotesize V} &\cdots \\
0&0&\bm{A}_{-1}^{\footnotesize V} & \bm{A}_{0}^{\footnotesize V} &\cdots \\
\vdots&\vdots&\vdots&\vdots&\ddots
 \end{bmatrix}
 \end{eqnarray}
 with
 \begin{eqnarray*}
 &&
\bm{A}_{-1}^{\footnotesize H}=\begin{bmatrix}
  h_{-1,0}& h_{-1,1}&0 &0&\cdots \\
q_{-1,-1} & q_{-1,0}&q_{-1,1} &0&\cdots \\
0&q_{-1,-1} & q_{-1,0}&q_{-1,1} &\cdots \\
0&0&q_{-1,-1} & q_{-1,0}&\cdots \\
\vdots&\vdots&\vdots&\vdots&\ddots
 \end{bmatrix}
 \ \text{and}\
 \bm{A}_{-1}^{\footnotesize V}=\begin{bmatrix}
  v_{0,-1}& v_{1,-1}&0 &0&\cdots \\
q_{-1,-1} & q_{0,-1}&q_{1,-1} &0&\cdots \\
0&q_{-1,-1} & q_{0,-1}&q_{1,-1} &\cdots \\
0&0&q_{-1,-1} & q_{0,-1}&\cdots \\
\vdots&\vdots&\vdots&\vdots&\ddots
 \end{bmatrix}
\\
&& \bm{A}_{0}^{\footnotesize H}=\begin{bmatrix}
 -h& h_{0,1}&0 &0&\cdots \\
q_{0,-1} & -q&0&0&\cdots \\
0&q_{0,-1} & -q&0&\cdots \\
0&0&q_{0,-1} &-q&\cdots \\
\vdots&\vdots&\vdots&\vdots&\ddots
 \end{bmatrix}
 \ \text{and}\
 \bm{A}_{0}^{\footnotesize V}=\begin{bmatrix}
 -v& v_{1,0}&0 &0&\cdots \\
q_{-1,0} & -q&0&0&\cdots \\
0&q_{-1,0} & -q&0&\cdots \\
0&0&q_{-1,0} &-q&\cdots \\
\vdots&\vdots&\vdots&\vdots&\ddots
 \end{bmatrix}
 \\
  && \bm{A}_{1}^{\footnotesize H}=\begin{bmatrix}
h_{1,0} & h_{1,1}&0 &0&\cdots \\
q_{1,-1} & 0&0 &0&\cdots \\
0&q_{1,-1} & 0&0 &\cdots \\
0&0&q_{1,-1} & 0 &\cdots \\
\vdots&\vdots&\vdots&\vdots&\ddots
 \end{bmatrix}
 \ \text{and}\
 \bm{A}_{1}^{\footnotesize V}=\begin{bmatrix}
v_{0,1} & v_{1,1}&0 &0&\cdots \\
q_{-1,1} & 0&0 &0&\cdots \\
0&q_{-1,1} & 0&0 &\cdots \\
0&0&q_{-1,1} & 0 &\cdots \\
\vdots&\vdots&\vdots&\vdots&\ddots
 \end{bmatrix}
 \end{eqnarray*}
 with $h=\sum_{s=-1}^1\sum_{t=-1}^1h_{s,t}=h_{1,0}+h_{1,1}+h_{0,1}+h_{-1,1}+h_{-1,0}$,   $v=\sum_{s=-1}^1\sum_{t=-1}^1v_{s,t}=v_{1,0}+v_{1,1}+v_{0,1}+v_{1,-1}+v_{0,-1}$,   and
$q=\sum_{s=-1}^1\sum_{t=-1}^1q_{s,t}=q_{-1,1}+q_{-1,0}+q_{-1,-1}+q_{0,-1}+q_{1,-1}$. One may write in a similar manner the matrices corresponding to level 0 denoted with $\bm{B}$.

\section{Stability condition\label{sec-3.1.1}}
In this section, we demonstrate how to extend the stability condition described on page \pageref{background} for the case $m=+\infty$. In particular,  in the following proposition we give an equivalent
sufficient and necessary stability condition in the case of infinite phases. Our proof relies on connecting the QBD drift condition derived by Neuts with the known drift condition for nearest neighbour random walks presented in  \cite[Theorem~1.2.1]{FIM}.
\begin{thm}[Theorem~1.2.1 \cite{FIM}] \label{ThmFIM}
For a homogeneous discrete time nearest neighbour random walk, with one step transition probabilities $\{p_{(n,m),(n',m')}\}$, ${(n,m),(n',m')\in \mathbb{Z}_+\times\mathbb{Z}_+}$, let
\begin{align*}
\bm{M}&=(M_x,M_y) = \Big(\sum\limits_{n',m'\geq 0 }(n'-n)p_{(n,m),(n',m')},\sum\limits_{n',m'\geq 0 }(m'-m)p_{(n,m),(n',m')}\Big),\ n,m>0;\\
\bm{M}'&=(M'_x,M'_y) = \Big(\sum\limits_{n',m'\geq 0 }(n'-n)p_{(n,0),(n',m')},\sum\limits_{n',m'\geq 0 }(m'-m)p_{(n,0),(n',m')}\Big),\ n>0;\\
\bm{M}''&=(M''_x,M''_y) = \Big(\sum\limits_{n',m'\geq 0 }(n'-n)p_{(0,m),(n',m')},\sum\limits_{n',m'\geq 0 }(m'-m)p_{(0,m),(n',m')}\Big),\ m>0.
\end{align*}
Then, when, $\bm{M}\neq0$, the homogeneous nearest neighbor random walk is ergodic if and only if, one of the following three conditions holds:
\begin{itemize}
\item[i)] $M_x<0$, $M_y<0$, $M_xM'_y-M_yM'_x<0$, and $M_yM''_x-M_xM''_y<0$;
\item[ii)] $M_x<0$, $M_y\geq0$, and $M_yM''_x-M_xM''_y<0$;
\item[iii)] $M_x\geq0$, $M_y<0$, and $M_xM'_y-M_yM'_x<0$.
\end{itemize}
\end{thm}
In the following proposition, using the above result, we prove that the QBD drift condition is necessary and sufficient for the ergodicity of the random walk, by showing how the QBD drift condition is connected to the results of \cite[Theorem~1.2.1]{FIM}.
\begin{prop}\label{stability_thm}
For the nearest neighbour random walk with generator matrices $\bm{G}^{\footnotesize V}$ or $\bm{G}^{\footnotesize H}$, defined in \eqref{generator_matrix}, the following hold:
\begin{itemize}
\item[i)] the continuous time Markov chain with generator $\bm{A}^{\footnotesize V}=\bm{A}_{-1}^{\footnotesize V}+\bm{A}_{0}^{\footnotesize V}+\bm{A}_{1}^{\footnotesize V}$ is ergodic if and only if $M_x<0$;
\item[ii)] the continuous time Markov chain with generator $\bm{A}^{\footnotesize H}=\bm{A}_{-1}^{\footnotesize H}+\bm{A}_{0}^{\footnotesize H}+\bm{A}_{1}^{\footnotesize H}$ is ergodic if and only if $M_y<0$;
\item[iii)] the QBD drift condition
 \begin{eqnarray*}
\bm{x}^{\footnotesize H}\bm{A}_{1}^{\footnotesize V}\mathds{1}<\bm{x}^{\footnotesize V}\bm{A}_{-1}^{\footnotesize V}\mathds{1},
 \end{eqnarray*}
with $\bm{x}^{\footnotesize V}$ the unique solution  to $\bm{x}^{\footnotesize V}\bm{A}^{\footnotesize V}=0$
 such that $\bm{x}^{\footnotesize V} \mathds{1}  =1$, where $\mathds{1}$ a column vector of ones, is equivalent to
$M_xM'_y-M_yM'_x<0$;
\item[iv)]  the QBD drift condition
 \begin{eqnarray*}
\bm{x}^{\footnotesize H}\bm{A}_{1}^{\footnotesize H}\mathds{1}<\bm{x}^{\footnotesize H}\bm{A}_{-1}^{\footnotesize H}\mathds{1},
 \end{eqnarray*}
with $\bm{x}^{\footnotesize H}$ the unique solution  to $\bm{x}^{\footnotesize H}\bm{A}^{\footnotesize H}=0$
 such that $\bm{x}^{\footnotesize H} \mathds{1}  =1$, where $\mathds{1}$ a column vector of ones, is equivalent to
$M_yM''_x-M_xM''_y<0$.
 \end{itemize}
\end{prop}


 \begin{proof}
 For the proof, we will first consider the uniformised (cf. \cite[Section~6.7]{Ross}) discrete time random walk and then use the drift conditions for random walks presented in \cite[Theorem~1.2.1]{FIM}. \\
For the conversion of the  random walk in the quadrant depicted in Figure \ref{fig:trd} to a DTMC, we simply need to consider the following transition probabilities
\begin{equation}\label{trans-prob}
p_{(n,m),(n',m')}=\frac{q_{(n,m),(n',m')}}{\sum\limits_{n',m'\geq 0 }q_{(n,m),(n',m')}},\ n,m, n',m'\geq0,
\end{equation}
where $q_{(n,m),(n',m')}$ denotes the transition rates from state $(n,m)$ to state $(n',m')\neq (n,m)$ for all states in the state space $\mathbb{Z}_+^2$. Of course, this might lead to issues with periodicity, but these can be avoided by allowing fictitious self-transitions (cf. \cite[Section~6.7]{Ross}). Such self-transitions are permitted in our setting due to the fact that the rates are homogeneous and can be therefore bounded by above. However, their introduction will only complicate the notation, and since the calculation of the drifts of Theorem \ref{ThmFIM} and the proof of the proposition are not effected by such self-transitions, we opt to not introduce them in the definition of the transition probabilities of the DTMC \eqref{trans-prob}.\\
For the discrete time random walk the statement of \cite[Theorem~1.2.1]{FIM} reads as follows:
\begin{align}
M_x<0 \ &\Leftrightarrow\ q_{1,-1}<q_{-1,-1}+q_{-1,0}+q_{-1,1},\label{M_condition_1}\\
M_y<0\ &\Leftrightarrow\ q_{-1,1}<q_{-1,-1}+q_{0,-1}+q_{1,-1},\label{M_condition_2}\\
M_xM'_y-M_yM'_x<0 &\Leftrightarrow\
\left[q_{1,-1}-(q_{-1,-1}+q_{-1,0}+q_{-1,1})\right]\left(h_{-1,1}+h_{0,1}+h_{1,1}\right)<\nonumber\\
&\qquad\quad \left[q_{-1,1}- (q_{-1,-1}+q_{0,-1}+q_{1,-1})\right]\left(h_{1,0}+h_{1,1}-h_{-1,1}-h_{-1,0}\right)
,\label{M_condition_3}\\
M_yM''_x-M_xM''_y<0&\Leftrightarrow\
\left[q_{-1,1}-(q_{-1,-1}+q_{0,-1}+q_{1,-1})\right]\left(v_{1,-1}+v_{1,0}+v_{1,1}\right)<\nonumber\\
&\qquad\quad \left[q_{1,-1}- (q_{-1,-1}+q_{-1,0}+q_{-1,1})\right]\left(v_{0,1}+v_{1,1}-v_{1,-1}-v_{0,-1}\right)
.\label{M_condition_4}
\end{align}
Similarly, after straightforward calculations of the four conditions of Proposition \ref{stability_thm} we can easily see that these four quantities match exactly the four drift conditions calculated in \eqref{M_condition_1}-\eqref{M_condition_4}, which concludes the proof of the proposition.
\end{proof}

\begin{remark}\label{remark3.1}
 Equivalently, $\bm{x}^{\footnotesize V}$
 corresponds to the vector of the invariant measure of a Birth-Death process with birth rates $\lambda_0^{\footnotesize V}=q_{1,-1}+q_{1,0}+q_{1,1}$, $\lambda_n^{\footnotesize V}=q_{1,-1}$, $n\geq1$, and death rates $\mu_n^{\footnotesize V}=q_{-1,-1}+q_{-1,0}+q_{-1,1}$, $n\geq 1$. So assertion {\em (i)} of Proposition \ref{stability_thm} can be simplified to $q_{1,-1}<q_{-1,-1}+q_{-1,0}+q_{-1,1}$.\\
 For the random walk at hand, $\bm{x}^{\footnotesize H}$
 corresponds to the vector of the invariant measure of a Birth-Death process with birth rates $\lambda_0^{\footnotesize H}=q_{-1,1}+q_{0,1}+q_{1,1}$, $\lambda_n^{\footnotesize H}=q_{-1,1}$, $n\geq1$, and death rates $\mu_n^{\footnotesize H}=q_{-1,-1}+q_{0,-1}+q_{1,-1}$, $n\geq 1$.  So assertion {\em (iii)} of Proposition \ref{stability_thm} can be simplified to $q_{-1,1}<q_{-1,-1}+q_{0,-1}+q_{1,-1}$.
 \end{remark}

\begin{remark}\label{remark_3.1}
Proposition \ref{stability_thm} can  be also extended, in a similar manner, to the case of nearest neighbour random walks with transitions to the North, North-East and East in the interior.
\end{remark}

\section{Equilibrium analysis: Related work}\label{sec-3}
\subsection{Matrix geometric approach\label{sec-3.1.2}}
Let $\pi_{n,m}$, $n,m\geq 0$ denote the equilibrium distribution of the QBD process. Then, if $\bm{\pi}_{n}=\begin{array}{c c c }( \pi_{n,0} & \pi_{n,1}& \cdots)\end{array}$, $n=0,1,\ldots$, denotes the equilibrium vector of level $n$,  it is known that
\begin{align}
\bm{\pi}_{n+1}=\bm{\pi}_{n}\bm{R},\ n=1,2,\ldots.\label{2a}
\end{align}
This last equation, by recursive application, yields the solution for the equilibrium vector as
\begin{align}
\bm{\pi}_{n}=\bm{\pi}_{1}\bm{R}^{n-1},\ n=1,2,\ldots.\label{2}
\end{align}
in terms of the matrix $\bm{R}$ and the vector of the equilibrium distribution corresponding to level $1$. Moreover, the infinite dimensional matrix $\bm{R}$ is obtained as the minimal non-negative solution to the matrix quadratic equation
\begin{eqnarray*}
\bm{A}_{1}^{\footnotesize H}+\bm{R}\bm{A}_{0}^{\footnotesize H}+\bm{R}^2\bm{A}_{-1}^{\footnotesize H}=0,
\end{eqnarray*}
see, e.g, \cite{Latouche,Neuts}.  Unfortunately, the structure of the random walk is overly generic and thus does not permit the calculation of the infinite matrix $\bm{R}$. This will be achieved by combining the two other approaches used in the analysis of random walks on the lattice: the compensation approach and the boundary value problem method.

\subsection{Compensation approach\label{sec-3.2}}
{\em The compensation approach} is developed by Adan et al. in a series of papers  \cite{MR1138205,MR1833660, MR1241929} and aims at a direct solution for the sub-class of two-dimensional random walks on the lattice of the first quadrant that obey the conditions for meromorphicity. The compensation approach can also be effectively used in  cases that the random walk at hand does not satisfy the aforementioned conditions, but the equilibrium distribution can still be written in the form of series of product-forms  \cite{Adan1,Adan2,Selen}.
This is due to the fact that this approach exploits the structure of the  equilibrium equations in the interior of the quarter plane by imposing that linear (finite or infinite) combinations of product-forms satisfy them. This leads to a kernel equation for the terms appearing in the product-forms. Then, it is required that these linear combinations satisfy the equilibrium equations on the boundaries as well. As it turns out, this can be done by alternatingly compensating for the errors on the two boundaries, which eventually leads to a (potentially) infinite series of product-forms.\\
For the model described in Section \ref{sec-2} one can easily show, cf. \cite[Chapter~2]{MR1138205}, that
\begin{description}
\item[Step 1:] $\pi_{n,m}= \alpha^n\beta^m$, $m,n>0$, is a solution to the equilibrium equations in the interior if and only if  $\alpha$ and $\beta$ satisfy the following kernel equation
\begin{eqnarray}
&&{\alpha}{\beta} (q_{-1,1}+q_{1,-1}+q_{0,-1}+q_{-1,-1}+q_{-1,0})=\nonumber\\
&&\qquad{\alpha}^2q_{-1,1}+{\beta}^2 q_{1,-1}+{\alpha}{\beta}^2q_{0,-1}+{\alpha}^2{\beta}^2q_{-1,-1}+{\alpha}^2{\beta} q_{-1,0}.\ \ \label{kernel_gen}
\end{eqnarray}
\item[Step 2:] Consider a product-form $c_0{\alpha}_0^n{\beta}_0^m$ that satisfies the kernel equation \eqref{kernel_gen} and also satisfies the equilibrium equations of the horizontal boundary. Without loss of generality we can assume that $c_0=1$. If the product-form $c_0{\alpha}_0^n{\beta}_0^m$ also satisfies the equilibrium equations of the vertical boundary then this constitutes the solution of the equilibrium equations up to a multiplicative constant that can be obtained using the normalising equation. Otherwise, consider a linear combination of two product-forms, say $c_0{\alpha}_0^n{\beta}_0^m+c_1{\alpha}^n{\beta}^m$, $m,n>0$, such that this combination satisfies now the  equilibrium equations of the vertical boundary. For this to happen it must be that
${\beta}={\beta}_0$ and then $\alpha={\alpha}_1$ is obtained as the solution of the  kernel equation \eqref{kernel_gen} for ${\beta}={\beta}_0$.
\item[Step 3:] Finally, as long as our expression of linear combinations of product-forms violates one of the two equilibrium equations on the boundary, we continue by adding new product-form terms satisfying  the kernel equation \eqref{kernel_gen}. This will eventually lead to Equations \eqref{1}-\eqref{pi0m}. Of course, one still needs to show that the series expression of Equations \eqref{1}-\eqref{pi0m} converge for all $n,m\geq 0$.
\end{description}
This procedure leads to the statement of Theorem \ref{Thm2.33}.

\subsection{Boundary value problem method\label{sec-3.3}}
{\em The boundary value problem method} is an analytic method which is applicable to some two-dimensional random walks restricted to the first quadrant. The bivariate probability generating function (PGF), say $$\Pi(x,y)=\sum_{n=0}^\infty\sum_{m=0}^\infty \pi_{n,m}x^ny^m,\ |x|,|y|\leq 1,$$ of the position of a homogeneous nearest neighbour random walk satisfies a functional equation of the form
\begin{align}
&K(x,y)\Pi(x,y)+A(x, y)\Pi(x, 0) + B(x, y)\Pi(0, y) + C(x, y)\Pi(0, 0)= 0,\label{funct}
\end{align}
with $K(x,y)$, $A(x,y)$, $B(x,y)$ and $C(x,y)$ known bivariate polynomials in $x$ and $y$, depending only on the  parameters of the random walk. In particular,
\begin{align*}
K(x,y)&=xy\left(  \sum_{s=-1}^1 \sum_{t=-1}^1 x^sy^tq_{s,t}-q\right),\\
&= y^2q_{-1,1}+x^2 q_{1,-1}+xq_{0,-1} +q_{-1,-1}+y q_{-1,0}  -xy (q_{-1,1}+q_{1,-1}+q_{0,-1}+q_{-1,-1}+q_{-1,0}),\\
A(x,y)&=-xy\left(  \sum_{s=-1}^1 \sum_{t=-1}^1 x^sy^t(q_{s,t}-h_{s,t})-(q-h)\right),\\
B(x,y)&=-xy\left(   \sum_{s=-1}^1 \sum_{t=-1}^1 x^sy^t(q_{s,t}-v_{s,t})-(q-v)\right),\\
C(x,y)&=xy\left(   \sum_{s=-1}^1 \sum_{t=-1}^1 x^sy^t(-q_{s,t}+h_{s,t}+v_{s,t}-r_{s,t})-(-q+h+v-r)\right),
\end{align*}
with $r= \sum_{s=-1}^1 \sum_{t=-1}^1 r_{s,t}$. \\
The boundary value problem method consists of the following steps:
\begin{itemize}\label{stepsBVP}
\item[i)] First, define the zero tuples $(x,y)$ such that $K(x,y)=0$, $|x|,|y|<1$.
\item[ii)] Then, along the curve $K(x,y)=0$ (and provided that $\Pi(x,y)$ is defined on this curve), Equation \eqref{funct} reads
\begin{align}
&A(x, y)\Pi(x, 0) + B(x, y)\Pi(0, y) + C(x, y)\Pi(0, 0)= 0.\label{boundary-problem-equation}
\end{align}
\item[iii)] Finally, in same instances, Equation \eqref{boundary-problem-equation} can be solved as a Riemann (Hilbert) boundary value problem.
\end{itemize}
Setting $K(1/{\alpha},1/{\beta})=0$ reduces to exactly Equation \eqref{kernel_gen}, indicating that the compensation approach and the boundary value problem method both utilise the same zero tuples. Furthermore, $A(1/{\alpha},1/{\beta})=0$ reduces to exactly the balance equations for the horizontal boundary satisfied by a product-form solution, i.e. $\pi_{n,m}=\alpha^n\beta^m$, $n,m\geq0$, and similarly $B(1/{\alpha},1/{\beta})=0$ reduces to exactly the balance equations for the vertical boundary.\\

Malyshev pioneered this approach of transforming the functional equation to a boundary value problem in the 1970's. The idea to reduce the functional equation for the generating function to a standard Riemann-Hilbert boundary value problem stems from the work of Fayolle and Iasnogorodski \cite{Fayolle_Iasnogorodski} on two parallel M/M/1 queues with coupled processors (the service speed of a server depends on whether
or not the other server is busy). Extensive treatments of the boundary value technique for functional equations can be found in Cohen and Boxma \cite[Part~II]{Cohen_Boxma} and Fayolle, Iasnogorodski and Malyshev \cite{FIM}. The model depicted in Figure \ref{fig:trd} can be analyzed by the approach developed by Fayolle and Iasnogorodski \cite{Fayolle_Iasnogorodski,FIM} and Cohen and Boxma \cite{Cohen_Boxma}, however this approach does not lead to the direct determination of the equilibrium distribution, since it requires inverting the PGF, and the existing numerical approaches for this method are  oftentimes tedious and case specific.

\section{Invariant measure properties}\label{sec-4}
In this paper, we explore the use of the matrix geometric approach by utilising ideas and results from the compensation approach and the boundary value problem method. Furthermore, through this work we achieve to connect the three approaches and gain valuable insight on the analytic and probabilistic interpretation of the terms appearing in the invariant measure. This connection will be achieved through the use of the PGF $\Pi(x,y)$ and the functional equation \eqref{funct} for the determination of the matrix $\bm{R}$. First and foremost, we show that the $\alpha$'s and $\beta$'s appearing in Theorem \ref{Thm2.33} are connected with the eigenvalues and left eigenvectors of matrix $\bm{R}$. This is established in the following proposition, that connects the derivation of the matrix $\bm{R}$ with the representation of the invariant measure as a series of product-forms, cf. Equation \eqref{mainInvM}, and hence the boundary value problem with the matrix geometric approach.

\begin{prop}\label{thm-1}
The terms $\{\alpha_k\}_{k\geq 0}$ constitute the different eigenvalues of the matrix $\bm{R}$.
For eigenvalue $\alpha_k$ the corresponding left eigenvector of the matrix $\bm{R}$ is
$\bm{p}_k=(p_{k,0}, p_{k,1}, p_{k,2}\ldots)$, with $p_{k,0}=e_k$, $k\geq0$,  $p_{0,m}=c_0\beta_{0}^m$, $m\geq1$, and
$p_{k,m}=c_k(\beta_{k-1}^m+f_k\beta_k^m)$, $k,m\geq1$,
if and only if $c_k\neq 0$, $k\geq0$.
\end{prop}

\begin{proof}
From \eqref{mainInvM} note that, for $n>0$,
\begin{eqnarray*}
\bm{\pi}_{n}&=&\begin{array}{c c c c }( \pi_{n,0} & \pi_{n,1}& \pi_{n,2}&\cdots)\end{array}\\
&=&\sum_{k=0}^\infty \alpha_k^{n}\bm{p}_k.
\end{eqnarray*}
Plugging this last result into \eqref{2a}, after straightforward manipulations yields
\begin{eqnarray}
\label{foralln}
\sum_{k=0}^\infty \alpha_k^{n}\bm{p}_k
(\alpha_k \bm{I}-\bm{R})=0,\ \forall n>0.
\end{eqnarray}
This last equation yields that
\begin{eqnarray*}
-\bm{p}_0 (\alpha_0 \bm{I}-\bm{R})=\sum_{k=1}^\infty \left(\alpha_k/\alpha_0\right)^{n}\bm{p}_k
(\alpha_k \bm{I}-\bm{R}) ,\ \forall n>0.
\end{eqnarray*}
Thus,
\begin{eqnarray}
0\leq ||\bm{p}_0 (\alpha_0 \bm{I}-\bm{R})||_\infty&\leq& \lim_{n\to\infty}\sum_{k=1}^\infty \left|\alpha_k/\alpha_0\right|^{n}
|| \bm{p}_k (\alpha_k \bm{I}-\bm{R})||_\infty \nonumber\\
&=& \sum_{k=1}^\infty \lim_{n\to\infty}\left|\alpha_k/\alpha_0\right|^{n}
|| \bm{p}_k (\alpha_k \bm{I}-\bm{R})||_\infty =0,
\label{p0eqaul0}
\end{eqnarray}
where the exchange of the series and the limit is allowed by the monotone convergence theorem, as $1>|\alpha_0|>|\alpha_1|>\cdots$, cf \cite[Page~33]{MR1138205}, and $|| \bm{p}_k (\alpha_k \bm{I}-\bm{R})||_\infty<1/|\alpha_k|$, $k\geq0$. The latter property is easily proven by Equation \eqref{foralln} for $n=1$. Therefore, recursively, it is evident that
\begin{eqnarray*}
\bm{p}_k(\alpha_k \bm{I}
-\bm{R})=0,\ \forall k\geq 0,
\end{eqnarray*}
which implies the statement of the proposition, cf. \cite{Shivakumar}.
\end{proof}

\begin{remark}
Note that  matrix $\bm{R}$ is crucial in deriving the first passage times for QBD's, see \cite{Neuts}. Thus the $\alpha$'s and $\beta$'s are connected with the first passage times.
\end{remark}

In the sequel, we present the properties of the spectrum and of the resolvent operator of the infinite matrix $\bm{R}$. Firstly, we present below the corresponding definitions.

\begin{defn}[Resolvent Operator, Spectrum] \label{defnRes}
Let $\mathcal{H}$ be a Hilbert space and let $\bm{R}$ be a linear operator on $\mathcal{H}$.
We write $\rho(\bm{R})$ for the set of all values $\alpha \in \mathbb{C}$ such that $(\alpha \bm{I} -\bm{R})$ is one-one, onto and for which $\mathrm{R}_{\bm{R}}(\alpha)=(\alpha \bm{I} -\bm{R})^{-1}\in \mathcal{L(H)}$, where $\mathcal{L(H)}$ is the Hilbert algebras with identity.
The map $\mathrm{R}_{\bm{R}}: \rho(\bm{R})\to \mathcal{L(H)}$ is called the resolvent operator.
\\
We denote the spectrum of $\bm{R}$ by $\sigma(\bm{R})=\mathbb{C}\setminus \rho(\bm{R})$.
\end{defn}
We would like to note that the spectrum is typically larger than the set of eigenvalues, except for the finite dimensional case of $\mathcal{H}=\mathbb{C}^{n\times n}$. In case of infinite dimension matrices, under suitable restrictions, one may show that the spectrum lies on the real line and (in general) is a spectral combination of a point spectrum of discrete eigenvalues with finite multiplicity, pure points, and a continuous spectrum of an  absolutely continuous part which is dense in $\mathcal{H}$, cf. \cite[Chapter V]{Lorch}. In special cases, e.g. compact, self adjoint operators, the spectrum consists only of the discrete values (together with 0 if the set of eigenvalues is infinite countable), see \cite[Corollary~9.14]{Hunter}. This has the property that we can define an orthonormal basis of $\mathcal{H}$ based on the eigenvectors of $\bm{R}$: The nonzero eigenvalues of $\bm{R}$ form a finite or countably infinite set $\{\alpha_k\}$ of real numbers and $\bm{R}=\sum_{k}\alpha_k\bm{P}_k$, where $\bm{P}_k$ is the orthonormal projection onto the eigenspace of the eigenvectors with eigenvalue $\alpha_k$. If the number of nonzero eigenvalues is countably infinite, then the series $\sum_{k}\alpha_k\bm{P}_k$ converges to $\bm{R}$ in the operator norm.\\

We formulate below a proposition that addresses a key point of the paper: the form of the resolvent operator, which will appear naturally when calculating the bivariate PGF. In addition, we establish that  the spectrum of the matrix $\bm{R}$ consists only of the discrete part of the eigenvalues $\{\alpha_k\}$, cf. Proposition \ref{thm-1}.

\begin{cor}\label{cor-1}
The resolvent operator of the matrix $\bm{R}$ can be calculated in  terms of the eigenvalues $\{\alpha_k\}_{k\geq 0}$ and the corresponding  eigenvectors $\{\bm{p}_k\}$ as follows
\begin{eqnarray}
\bm{\pi}_1\,(\alpha\bm{I}- \bm{R})^{-1}=\sum_{k=0}^\infty \frac{\alpha_k }{\alpha-\alpha_k } \bm{p}_k.\label{Resolvent}
\end{eqnarray}
\end{cor}
\begin{proof}
Firstly, note that
\begin{align*}
\Pi(x,y)&=\sum_{n=0}^\infty x^n\, \bm{\pi}_{n}\, (\begin{array}{c c c c }1 & y&y^2& \cdots\end{array})^{\small T}\nonumber\\
&= (\begin{array}{c c c c }\bm{\Pi}_{\mbox{\footnotesize{-}},0}(x) & \bm{\Pi}_{\mbox{\footnotesize{-}},1}(x)&\bm{\Pi}_{\mbox{\footnotesize{-}},2}(x)& \cdots\end{array})\, (\begin{array}{c c c c }1 & y&y^2& \cdots\end{array})^{\small T},
\end{align*}
with $\bm{\Pi}_{\mbox{\footnotesize{-}},m}(x)=\sum_{n=0}^\infty x^n\, {\pi}_{n,m}$, $m\geq0$.
On the one hand, from the  definition of $\bm{\Pi}_{\mbox{\footnotesize{-}},m}(x)$ and Equation \eqref{mainInvM}, we have, for $n>0$,
\begin{eqnarray}
 (\begin{array}{c c c c }\bm{\Pi}_{\mbox{\footnotesize{-}},0}(x) & \bm{\Pi}_{\mbox{\footnotesize{-}},1}(x)&\bm{\Pi}_{\mbox{\footnotesize{-}},2}(x)& \cdots\end{array})&=& (\begin{array}{c c c c }\sum_{n=0}^\infty x^n\, {\pi}_{n,0}& \sum_{n=0}^\infty x^n\, {\pi}_{n,1}&\sum_{n=0}^\infty x^n\, {\pi}_{n,2}& \cdots\end{array})\nonumber\\
&=&\bm{\pi}_{0}+\sum_{n=1}^\infty x^n \sum_{k=0}^\infty \alpha_k^n \bm{p}_k\nonumber\\
&=&\bm{\pi}_{0}+\sum_{k=0}^\infty \frac{\alpha_k }{x^{-1}-\alpha_k } \bm{p}_k,\label{Pi-1}
\end{eqnarray}
where $\bm{p}_k=(p_{k,0}, p_{k,1}, p_{k,2}\ldots)$, with $p_{k,0}=e_k$, $k\geq0$,  $p_{0,m}=c_0\beta_{0}^m$, $m\geq0$,
$p_{k,m}=c_k(\beta_{k-1}^m+f_k\beta_k^m)$, $k,m\geq1$.\\
On the other hand, from Equation \eqref{2}, we obtain, for $n>0$,
\begin{eqnarray}
 (\begin{array}{c c c c }\bm{\Pi}_{\mbox{\footnotesize{-}},0}(x) & \bm{\Pi}_{\mbox{\footnotesize{-}},1}(x)&\bm{\Pi}_{\mbox{\footnotesize{-}},2}(x)& \cdots\end{array})&=&\sum_{n=0}^\infty x^n\, \bm{\pi}_{n}\nonumber\\
&=&\bm{\pi}_{0}+\sum_{n=1}^\infty x^n\, \bm{\pi}_{1}\,\bm{R}^{n-1}\nonumber\\
&=&\bm{\pi}_{0}+\bm{\pi}_1\,(x^{-1}\bm{I}- \bm{R})^{-1}.\label{Pi-2}
\end{eqnarray}
Note that Equations \eqref{Pi-1} and \eqref{Pi-2} are two different representations of the same vector of partial PGFs and are therefore equal. Setting $x^{-1}=\alpha$ in these two representations concludes the proof of the corollary.
\end{proof}
\begin{remark}\label{remark-As2}
Note that the proof of Corollary \ref{cor-1} justifies the choice we made in Assumption \ref{Assumptions2}, to set $N=0$.
\end{remark}

From the form of the resolvent operator \eqref{Resolvent}, in combination with known properties of the  eigenvalues $\{\alpha_k\}_{k\geq 0}$ and the corresponding  eigenvectors $\{\bm{p}_k\}$, we can now establish the properties of the spectrum of matrix $\bm{R}$.

\begin{thm}\label{MainThm}
Under $\bm{\pi}_1$, the spectral properties of $\bm{R}$ are:
\begin{itemize}
\item[i)] Every nonzero $\alpha_k\in \sigma(\bm{R})$ is an eigenvalue of $\bm{R}$, i.e. the spectrum consists only of the eigenvalues and does not have a continuous part.
\item[ii)] $\sigma(\bm{R})$ is at most countably infinite.
\item[iii)] The eigenvalues only accumulate at 0. Furthermore, $0\in\sigma(\bm{R})$ (since the dimension of $\bm{R}$ is not finite).
\item[v)] Every nonzero $\alpha_k\in \sigma(\bm{R})$ is a pole of the resolvent function $\alpha\to (\alpha\bm{I}- \bm{R})^{-1}$.
\end{itemize}
\end{thm}
\begin{proof}
For the proof of the theorem it is necessary to show that
\begin{itemize}
\item[i)] $\{\alpha_k\}$ and $\{\beta_k\}$ tend to zero as $k$ tends to infinity, cf \cite[Page~33]{MR1138205};
\item[ii)] $|\alpha_k|<1$ and $|\beta_k|<1$, cf \cite[Page~33]{MR1138205};
\item[iii)]For every fixed $\alpha$, with $0<|\alpha|<1$, there is exactly one $\beta$, with $0<|\beta|<|\alpha|$ and one root $\beta$ with $|\beta|>|\alpha|$. The same holds for $\alpha$ and $\beta$ interchanged, cf. \cite[Lemma~2.7]{MR1138205}.
\end{itemize}
Combining the above results, with the form of the resolvent, the proof of the theorem follows.
\end{proof}

\begin{remark}
It is worthy to mention some of the spectral properties arising in the case of random walk with an invariant measure representable with a finite sum of product-forms, $\pi_{n,m}=\sum_{k=1}^K c_k \alpha_k^n \beta_k^m$, $n,m\geq1$, such as the Jackson network of two stations in tandem ($K=1$). In these cases, similarly to our case one may prove that $\{\alpha_k\}$ are the eigenvalues of the matrix $\bm{R}$ and $\{\beta_k\}$ appear in the definition of the eigenvectors. Considering that now there is a finite number of real eigenvalues of matrix $\bm{R}$, which can be used to create the eigenvectors of the orthonormal basis of $\mathcal{H}$, it is needed to consider for the basis also the eigenvectors arising with eigenvalue zero, cf. the proof in \cite[Theorem~9.16]{Hunter}. In this sense, eigenvalue zero has an infinite multiplicity. In addition the spectrum may contain an absolutely continuous part. Thus, truncating the state space of the phase and considering a finite matrix $\bm{R}$ might lead  to the issues noticed in \cite{Taylor,Latouche1,Motyer}.
\end{remark}

\section{Calculation of the invariant measure}\label{sec-6}
In this section we turn our focus to the recursive calculation of the ${\alpha}$'s and ${\beta}$'s. To this purpose, we first compute the bivariate PGF in terms of the resolvent of the matrix $\bm{R}$.
\begin{eqnarray}
\Pi(x,y)&=&\sum_{n=0}^\infty x^n\, \bm{\pi}_{n}\, (\begin{array}{c c c c }1 & y&y^2& \cdots\end{array})^{\small T}\nonumber\\
&=&\left( \bm{\pi}_{0}+\sum_{n=1}^\infty x^n\, \bm{\pi}_{1}\,\bm{R}^{n-1}\right)\, (\begin{array}{c c c c }1 & y&y^2& \cdots\end{array})^{\small T}\nonumber\\
&=&\left( \bm{\pi}_{0}+\bm{\pi}_1\,(x^{-1}\bm{I}- \bm{R})^{-1} \right)(\begin{array}{c c c c }1 & y&y^2& \cdots\end{array})^{\small T}.\label{PGF-QBD}
\end{eqnarray}
Substituting the above result in the functional equation \eqref{funct} yields after some straightforward manipulations
\begin{align*}
&\bm{\pi}_1\,(x^{-1}\bm{I}- \bm{R})^{-1} \Big( K(x,y) (\begin{array}{c c c c }1 & y&y^2& \cdots\end{array})^{\small T} +A(x, y) (\begin{array}{c c c c }1 & 0&0& \cdots\end{array})^{\small T}\Big) \nonumber\\
&\quad =-\,\bm{\pi}_0\,\Big( \left(K(x,y)+B(x, y) \right)(\begin{array}{c c c c }1 & y&y^2& \cdots\end{array})^{\small T} +\left(A(x,y)+C(x, y)\right) (\begin{array}{c c c c }1 & 0&0& \cdots\end{array})^{\small T}\Big).
\end{align*}
Using the result of Corollary \ref{cor-1} in the above expression immediately yields
\begin{align*}
&\sum_{k=0}^\infty \frac{\alpha_k }{x^{-1}-\alpha_k } \bm{p}_k \Big( K(x,y) (\begin{array}{c c c c }1 & y&y^2& \cdots\end{array})^{\small T} +A(x, y) (\begin{array}{c c c c }1 & 0&0& \cdots\end{array})^{\small T}\Big) \nonumber\\
&\quad =-\,\bm{\pi}_0\,\Big( \left(K(x,y)+B(x, y) \right)(\begin{array}{c c c c }1 & y&y^2& \cdots\end{array})^{\small T} +\left(A(x,y)+C(x, y)\right) (\begin{array}{c c c c }1 & 0&0& \cdots\end{array})^{\small T}\Big).
\end{align*}
Equivalently, by defining $\bm{e}_j$ to be a infinite dimension column vector with a 1 in the $i$-th position and 0 elsewhere, the above equation reduces to
\begin{align*}
&\sum_{k=0}^\infty \frac{\alpha_k }{x^{-1}-\alpha_k } \bm{p}_k \Big( K(x,y) \sum_{i=1}^\infty y^{i-1}\bm{e}_i+A(x, y) \bm{e}_1\Big)=-\,\bm{\pi}_0\,\Big( \left(K(x,y)+B(x, y) \right)\sum_{i=1}^\infty y^{i-1}\bm{e}_i+\left(A(x,y)+C(x, y)\right)  \bm{e}_1 .
\end{align*}
After straightforward calculations the above can be equivalently written as
\begin{align}\label{funct-MAM}
&\sum_{k=0}^\infty \frac{\alpha_k }{x^{-1}-\alpha_k } \sum_{i=1}^\infty ( K(x,y)  y^{i-1}p_{k,i-1}+A(x, y) p_{k,0}\delta_{\{i=1\}}) \nonumber\\
&\quad =-\sum_{i=1}^\infty\Big( \left(K(x,y)+B(x, y) \right) y^{i-1}\pi_{0,i-1}+\left(A(x,y)+C(x, y)\right)\pi_{0,0}\delta_{\{i=1\}}\Big) ,
\end{align}
with $\delta_{\{\cdot\}}$ an indicator function taking value 1 if the event $\{\cdot\}$ is satisfied and 0 otherwise.\\

Note that, we can meromorphically continue the bivariate PGF on the entire complex domain,  i.e., the bivariate PGF has a finite number of poles in every finite domain, cf. \cite{Cohen-asq}. More concretely, the PGF is holomorphic on the entire complex domain except for a set of isolated points (the poles of the function) $x^{-1}={\alpha}_k$ and $y^{-1}={\beta}_k$, $k\geq0$. \\

We use Equation \eqref{funct-MAM} to show for self-completeness an alternative way in recursively calculating the terms appearing in the expressions for the invariant measure, cf. Theorem \ref{Thm2.33}. Firstly, we describe in the following paragraph how to  recursively obtain the sequence of $\alpha$'s. Thereafter, we present an iterative approach for the calculation of the eigenvectors of the matrix $\bm{R}$.

\subsection{Recursive calculation of the eigenvalues}\label{sec-6.1}
For the recursive calculation of the $\alpha$'s and $\beta$'s, we use the first step of the boundary value problem (see assertion i) on page \pageref{stepsBVP}). There the main idea lies on defining the zero tuples $(x,y)$, inside $|x|,|y|<1$, such that the kernel $K(x,y)=0$ of the functional equation \eqref{funct-MAM} becomes zero. To this purpose, we set $x^{-1}={\alpha}_0$ in \eqref{funct-MAM}. Note that, for  $|y|<1$, the right hand side of \eqref{funct-MAM} is well defined, which implies that
\begin{align*}
\sum_{i=1}^\infty \big( K({\alpha}_0^{-1},y)  y^{i-1}p_{0,i-1}+A({\alpha}_0^{-1}, y)p_{0,0}\delta_{\{i=1\}} \big)=0.
\end{align*}
%
%
%
Restricting the investigation on the set of $y$-roots that satisfy $K({\alpha}_0^{-1},y) =0$, the above equation yields that $A({\alpha}_0^{-1},y)=0$, since $p_{0,0}\neq0$. Thus choosing the $\alpha_0$ such that $K({\alpha}_0^{-1},y) =0$ and $A({\alpha}_0^{-1},y)=0$ reveals the starting solution ${\alpha}_0$, with $|{\alpha}_0|<1$, for the iterative calculation of the sequences $\{{\alpha}_k\}_{k\geq 0}$ and $\{{\beta}_k\}_{k\geq 0}$.\\

The existence of the solution ${\alpha}_0$ inside the unit disk is proved in \cite[Theorem~2.19]{MR1138205} by considering a random walk that exhibits the same behaviour in the interior and the horizontal boundary, see \cite[Figure~2.5,~page 41]{MR1138205}. More concretely, under Assumptions \ref{Assumptions} and the stability condition, it is proven in \cite[Theorem~2.19]{MR1138205} that:
\begin{itemize}
\item[i)] if $h_{1,1}>0$, then there are two feasible ${\alpha}_0$-roots, one in $(0,1)$ and the other in in $(-1,0)$;
\item[ii)] if $h_{1,1}=0$, then there is one feasible ${\alpha}_0$-root, located in $(0,1)$.
\end{itemize}
Furthermore, it is also proven that choosing ${\alpha}_0\in(0,1)$ guarantees that the created series expression of product-forms for the equilibrium distribution converges, see \cite[Theorem~2.33]{MR1138205}.\\

For the starting solution $x^{-1}={\alpha}_0$ we  calculate recursively $y^{-1}={\beta}_0$ by the kernel equation $K(x,y)=0$. This will produce a single ${\beta}$ with $|{\beta}|< |{\alpha}|$, see the proof of Theorem \ref{MainThm}. \\

Note that we cannot start in the reverse order, that is, first calculate ${\beta}_0$ and thereafter $\alpha_0$, since the system $K(x,y)=0$ and $B(x,y)=0$ does not necessarily have a $y^{-1}(={\beta}_0)$-root inside the unit disk. The proof of this statement follows simply by constructing a random walk which exhibits in the interior and on the vertical boundary the same behaviour as the original random walk and showing that this new walk is non-ergodic. The constructed new random walk follows the same reasoning as in the proof of the existence of the starting ${\alpha}_0$, see, e.g., \cite[Figure~2.5]{MR1138205} or \cite[Figure~2]{Adan2}.\\

We can proceed in an analogous manner and construct recursively the entire set of product-form terms setting $K(\alpha^{-1},\beta^{-1})=0$. This will produce the entire sequence $\{\alpha_k\}$ and $\{\beta_k\}$, see Figure~\ref{fig:alpha_beta} for an illustration of the evolution of the $\alpha$ and $\beta$ terms.\\

\begin{figure*}[hb]%
\centering%
\includegraphics{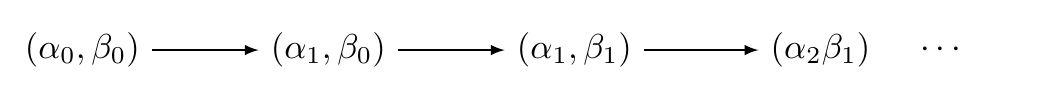}
\caption{The recursive structure of the product-form terms.}%
\label{fig:alpha_beta}%
\end{figure*}%

We can explore, similarly to the boundary value problem method, the fact that the kernel equation $K(x,y)=0$ is a quadratic function in either $y$ or $x$. This allows us to re-write the kernel function as follows
\begin{align*}
&K(x,y)=a_x(y-y_+(x))(y-y_-(x)),\\
&K(x,y)=\tilde{a}_y(x-x_+(y))(x-x_-(y))
\end{align*}
for some functions $a_x$ and $\tilde{a}_y$. This  yields
\begin{align}
&y_+\left({\alpha_k}^{-1}\right)={\beta_k}^{-1},\qquad\quad x_+\left({\beta_k}^{-1}\right)={\alpha_{k+1}}^{-1},\label{y+}\\
&y_-\left({\alpha_k}^{-1}\right)={\beta_{k-1}}^{-1},\qquad x_-\left( {\beta_k}^{-1}\right)={\alpha_k}^{-1}.\label{y-}
\end{align}
The above equations produce explicitly the sequence of the zero-tuples of the kernel equation, cf. Figure~\ref{fig:alphabetaunique}, with 'forward' operators $y_+$ and $x_+$
and 'backward' operators $y_-$ and $x_-$ constructed as the solutions of the kernel equation $K(x,y)=0$ with respect to $y$ and $x$.

\begin{figure*}[h!]%
\centering%
\includegraphics{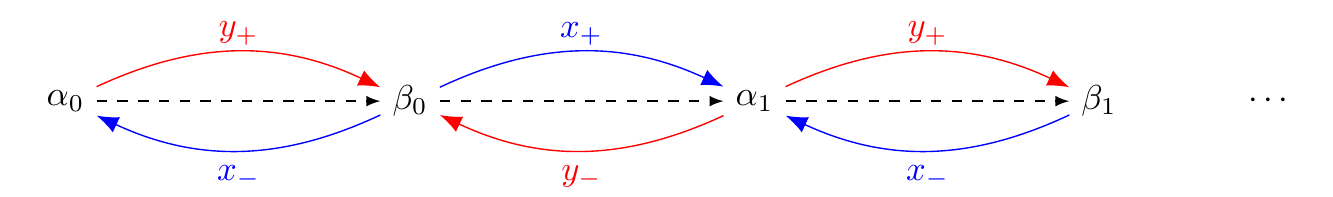}
\caption{Evolutions of ${\alpha}$'s and ${\beta}$'s.}%
\label{fig:alphabetaunique}%
\end{figure*}%

\subsection{Recursive calculation of the eigenvectors}\label{sec-6.2}
It remains to show how to calculate the eigenvectors of the matrix $\bm{R}$. Before we can calculate the eigenvectors, we will first determine $\bm{\pi}_0$. The reason behind this choice is that the elements of the vector $\bm{\pi}_0$ appear on the right hand side of \eqref{funct-MAM}.
First, we set $y=y_+(x)$ in \eqref{funct-MAM} (for which $K(x, y_+(x))=0$), this yields
\begin{align*}
&\sum_{k=0}^\infty \frac{\alpha_k }{x^{-1}-\alpha_k } A(x, y_+(x)) p_{k,0}  =-\sum_{i=1}^\infty\Big( B(x, y_+(x))  y_+(x)^{i-1}\pi_{0,i-1}+\big(A(x,y_+(x))+C(x, y_+(x))\big)\pi_{0,0}\delta_{\{i=1\}}\Big).
\end{align*}
Note that $\pi_{0,i-1}$ can be substituted by its series expression \eqref{mainInvM}, this permits us to rewrite the above equation as
\begin{align*}
0=\sum_{k=0}^\infty \frac{\alpha_k }{x^{-1}-\alpha_k } A(x, y_+(x)) p_{k,0}  + B(x, y_+(x))  \sum_{k=0}^\infty d_k \frac{\beta_k}{y_+(x)^{-1}-\beta_k}+\big(A(x,y_+(x))+B(x, y_+(x))+C(x, y_+(x))\big)\pi_{0,0}.
\end{align*}
We multiply the resulting equation with $x^{-1}-\alpha_k$ and take the limit as $x\rightarrow \alpha_k^{-1}$. This yields
\begin{align}
0=\alpha_k  A\left({\alpha_{k}}^{-1},{\beta_{k}}^{-1}\right)p_{k,0}
+B\left({\alpha_{k}}^{-1},{\beta_{k}}^{-1}\right)d_{k}\beta_k\lim\limits_{x\to {\alpha_{k}}^{-1}}
\frac{x^{-1}-\alpha_{k}}{y_+(x)^{-1}-\beta_{k}}.\label{coeff1}
\end{align}
Similarly, repeating the above procedure for $y=y_-(x)$ and as $x\to\alpha_{k+1}^{-1}$ produces
\begin{align}
0=\alpha_{k+1}  A\left({\alpha_{k+1}}^{-1},{\beta_{k}}^{-1}\right)p_{k+1,0}
+B\left({\alpha_{k+1}}^{-1},{\beta_{k}}^{-1}\right)d_{k}\beta_{k}\lim\limits_{x\to {\alpha_{k+1}}^{-1}}
\frac{x^{-1}-\alpha_{k+1}}{y_-(x)^{-1}-\beta_{k}}.\label{coeff2}
\end{align}
Instead, we could have chosen $x=x_\pm(y)$,
in which case the resulting equations would be identical to the ones derived above.\\

Now starting from $p_{0,0}$ all the elements $\{p_{k,0}\}_{k\geq0}$ and $\{d_{k}\}_{k\geq0}$
are obtained recursively as follows: for a given $p_{k,0}$ using Equation \eqref{coeff1} one can
derive $d_k$, next given $d_k$ Equation \eqref{coeff2} produces
 $p_{k+1,0}$.\\

 Having  recursively identified in terms of $p_{0,0}$ the elements $\{p_{k,0}\}_{k\geq0}$ and $\{d_{k}\}_{k\geq0}$, we show in the sequel how to obtain the unknown terms $\{p_{k,i}\}$, $j\geq1$. To this purpose, we set $x^{-1}={\alpha}_k$ in \eqref{funct-MAM} (equivalently, we can multiply with $x^{-1}-\alpha_k$ and take the limit as $x\rightarrow \alpha_k^{-1}$). Note that, for  $|y|<1$, the right hand side of \eqref{funct-MAM} is well defined, which implies that
\begin{align}\label{eigenvectors}
K({\alpha}_k^{-1},y)\sum_{i=1}^\infty  y^{i}p_{k,i}+\big(  A({\alpha}_k^{-1}, y)+K({\alpha}_k^{-1},y)\big)p_{k,0}=0,\ |y|<1.
\end{align}
Taking derivative in $y$ in \eqref{eigenvectors} and setting $y=0$ yields
\begin{align}
p_{k,1}&=-\left.\frac{\mathrm{d}\big(  A({\alpha}_k^{-1}, y)+K({\alpha}_k^{-1},y)\big)}{\mathrm{d}y}\right|_{y=0}\frac{p_{k,0}}{K({\alpha}_k^{-1},0)}\nonumber\\
&= - \frac{h_{-1,0}+\alpha_k^{-2}h_{1,0}-\alpha_k^{-1} h}{\alpha_k^{-2}q_{1,-1}+\alpha_k^{-1}q_{0,-1}+q_{-1,-1}}\,p_{k,0}
.\label{eigenvector0}
\end{align}
Taking second derivative in $y$ in \eqref{eigenvectors} and setting $y=0$ yields
\begin{align}
p_{k,2}&=-\left.\frac{\mathrm{d}^2\big(  A({\alpha}_k^{-1}, y)+K({\alpha}_k^{-1},y)\big)}{\mathrm{d}y^2}\right|_{y=0}\frac{p_{k,0}}{K({\alpha}_k^{-1},0)}-
2\left.\frac{\mathrm{d}\big(  K({\alpha}_k^{-1},y)\big)}{\mathrm{d}y}\right|_{y=0}\frac{p_{k,1}}{K({\alpha}_k^{-1},0)}
\nonumber\\
&= -\frac{2h_{-1,1}+2\alpha_k^{-1}h_{0,1}+2\alpha_k^{-2} h_{1,1}}{\alpha_k^{-2}q_{1,-1}+\alpha_k^{-1}q_{0,-1}+q_{-1,-1}}\,p_{k,0}
 -2\frac{q_{-1,-1}+\alpha_k^{-1}q_{0,-1}+\alpha_k^{-2}q_{1,-1}}{\alpha_k^{-2}q_{1,-1}+\alpha_k^{-1}q_{0,-1}+q_{-1,-1}}\,p_{k,1}
.\label{eigenvector1}
\end{align}
Taking higher derivatives in $y$ in \eqref{eigenvectors} and setting $y=0$ yields
\begin{align}
0&=\binom{i}{0}i!p_{k,i}K(\alpha_k^{-1},0)+
\binom{i}{1}(i-1)!p_{k,i-1}\left.\frac{\mathrm{d}\big(  K({\alpha}_k^{-1},y)\big)}{\mathrm{d}y}\right|_{y=0}+
\binom{i}{2}(i-2)!p_{k,i-2}\left.\frac{\mathrm{d}^2\big(  K({\alpha}_k^{-1},y)\big)}{\mathrm{d}y^2}\right|_{y=0}
\nonumber\\
&=i!(q_{-1,-1}+\alpha_k^{-1}q_{0,-1}+\alpha_k^{-2}q_{1,-1})\,p_{k,i} +
i!\,q_{-1,0}\, p_{k,i-1}+
i!\,q_{-1,1}\, p_{k,i-2}
.\label{eigenvectori}
\end{align}
Note that for the construction of the above system it is not needed to choose $y=0$, but we can instead choose any $|y|\leq 1$, however the choice of $y=0$ simplifies significantly the calculations.

\subsection{Description of the algorithm}\label{sec-6.3}
The iterative procedure is as follows:
\begin{description}
\item[Step 1)] Derive the $\alpha$'s and $\beta$'s: First obtain $\alpha_0$ by solving the system
\[A(\alpha_0^{-1},y)=0\text{ and }K(\alpha_0^{-1},y)=0.\]
Then, calculate iteratively the rest of the $\alpha$ and $\beta$ terms as depicted in Figure \ref{fig:alphabetaunique} and according to the functions $y_+$ and $x_+$ defined in \eqref{y+}.
\item[Step 2)] Starting from $p_{0,0}$, we  derive $d_0$ from Equation \eqref{coeff1}. Thereafter, use $d_0$ to calculate $p_{1,0}$ from Equation \eqref{coeff2}. Continue in this manner and determine recursively the sequences $\{d_k\}$ and $\{p_{k,0}\}$.
\item[Step 3)] Determine $p_{k,i}$, for $i=1,2,\ldots$, from $p_{k,0}$, using Equations \eqref{eigenvector0}--\eqref{eigenvectori}.
\item[Step 4)] The starting constant $p_{0,0}$ is uniquely determined by the normalisation equation. Equivalently, one may opt to set $p_{0,0}=1$ and normalise at the end the eigenvectors using the normalisation equation.
\end{description}

\subsection{Numerical evaluation of matrix $\bm{R}$}\label{sec-6.4}
It is known, see  \cite{MR1138205,CohenJAP}, that the sequences $\{\alpha_k\}_{k\geq0}$ and $\{\beta_k\}_{k\geq0}$ decrease exponentially fast to $0$.
Based on this fact, we suggest to truncate the dimension of the matrix $\bm{R}$, say at phase $N$,
and obtain its approximation, say $\bm{R}_N$, as
\[\bm{R}_N=\bm{P}_N^{-1}\bm{D}_N\bm{P}_N,\]
with $\bm{D}_N={\rm diag}(\alpha_0, \alpha_1,\ldots, \alpha_{N-1})$ and
the matrix $\bm{P}_N=(\bm{P}_0^{(N)},\ldots\bm{P}_{N-1}^{(N)})$, where $\bm{P}_k^{(N)}=(p_{k,0},\ldots, p_{k,N-1})$.
Then, as $N\rightarrow\infty$ the matrix $\bm{R}_N=\bm{P}_N\bm{D}_N\bm{P}_N^{-1}$ converges to the infinite matrix $\bm{R}$, cf. \cite{Shivakumar}.    Furthermore, closely inspecting the structure of the matrix $\bm{P}_N=(p_{k,m})_{0\leq k,m\leq N-1}$ we observe that it can be written as a generalised Vandermonde matrix for which the inverse can be easily calculated.\\

The idea of spectral truncation proposed in the above paragraph, due to the ``diagonisability'' of matrix $\bm{R}$, is equivalent to truncating the series expression of the equilibrium distribution by considering only the first $N$ product-form terms. It is known, see e.g. \cite{MR1138205}, that these series can be viewed as asymptotic expansions, thus including a higher number of terms of the series improves the approximation of the result.\\

There are several numerical procedure that one may implement in order to calculate the eigenvalues and eigenvectors of an infinite matrix, see, e.g. \cite{QR,QR2} and the references therein.

\section{Example: the join the shortest queue model}\label{sec-7}

We consider a queueing system with a Poisson arrival process with rate $2\rho$, two identical exponential servers, both with rate 1, and join the shortest queue (JSQ) routing, i.e. customers upon arrival join the queue with the smallest number of customers and in the case of a tie they choose either queue with probability 1/2. Such a queueing system can be modelled as a Markov process with states $(q_1,q_2) \in \mathbb{Z}_+^2$, where $q_i$ is the number of customers at queue $i$, including a customer possibly in service. By defining $m = \min(q_1,q_2)$ and $n = q_2 - q_1$, one transforms the state space from an inhomogeneous random walk in the quadrant to a homogeneous random walk in the half plane, where the two quadrants are symmetrical.
The transition rate diagram of the Markov process for $n,m\geq0$ is shown in Figure~\ref{fig:JSQ_various_models}.
Implementing Proposition \ref{stability_thm} one can easily prove the stability condition $\rho<1$.

\begin{figure}[h!]%
\centering%
\includegraphics{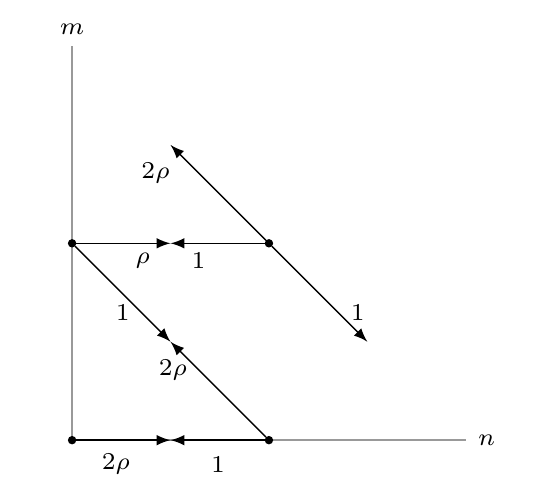}
\caption{Transition rate diagram on the state space $(n,m)$ for the symmetric JSQ system with two servers. Only the transitions at a few selected states are depicted as an indication.}%
\label{fig:JSQ_various_models}%
\end{figure}%

For the symmetric JSQ model one can easily show
\begin{align*}
K(x,y) &=x(2\rho x+1)-2(1+\rho)xy+y^2,\\
A(x,y)&=y^2-(1+2\rho)x y+ x,\\
B(x,y)&=x(1+2\rho)(x-y).
\end{align*}
Implementing  steps 1-6 described Section \ref{sec-4} permits us to reconstruct the
results in \cite[Chapter~3]{MR1138205} and \cite{Kingman} regarding the calculation of the equilibrium distribution, the results in \cite{Gertsbakh} regarding the numerical calculation for the matrix $\bm{R}$ and the results in \cite{Li} regarding the geometric decay rate given by the spectral radius of $\bm{R}$, see \cite{Taylor}.
In particular, solving $K(x,y)=0$ and $A(x,y)=0$ we obtain the single solution $x=\alpha_0^{-1}\in(1,\infty)$: $x=1/\rho^2$. Hence, we choose as a starting solution $\alpha_0=\rho^2\in(0,1)$. From the kernel equation $K(x,y)$ we can obtain for $x=1/\rho^2$ the starting $\beta$-solution: $\beta_0=\rho^2/(2+\rho)\in(0,1)$. Using recursively the kernel equation we can reconstruct the entire sequence $\{\alpha_k\}_{k\geq0}$ and $\{\beta_k\}_{k\geq0}$. The  calculation of the entire equilibrium distribution representation matches perfectly the results in
\cite[Chapter~3]{MR1138205}, where the JSQ model is treated using the compensation approach. \\
Note that solving $K(x,y)=0$ and $A(x,y)=0$ yields a unique $x$-solution outside the unit disk, or equivalently a unique  $\alpha$-solution in the unit disk. However solving $K(x,y)=0$ and $B(x,y)=0$ does not  yield any  $y$-solution outside the unit disk, or equivalently a unique  $\beta$-solution in the unit disk. In particular, the system
$K(x,y)=0$ and $B(x,y)=0$ yields the following solutions
$$\{x\to 0,y\to 0\},\{x\to 1,y\to 1\},\left\{x\to -\frac{1}{2 \rho },y\to 0\right\}$$
all of which in terms of $y$ lie in the interior or on the closure of the unit disk. This is related to the fact that the invariant measure cannot be written as
\[\pi_{n,m}=\beta^m f(n),\]
for some function of $n$. This is due to the ordering of the $\alpha$ and $\beta$ as $1<|\alpha_0|<|\beta_0|<|\alpha_1|<|\beta_1|<\cdots$.

For a numerical comparison between the results proposed in this paper and the derivation of the matrix $\bm{R}$ by increasing the state-space truncation the interested reader is refereed to \cite{Gertsbakh}.

\section{Conclusions and future work}\label{Conclusions}
In this paper, we generalise the QBD drift conditions for random walks in the lattice and illustrate the connection between the form of the equilibrium distribution depicted in Equation \eqref{1} and the derivation of the eigenvalues and eigenvectors of the infinite matrix $\bm{R}$. Moreover, this work sets the groundwork for the probabilistic interpretation of the terms $\alpha$ and $\beta$ appearing in the series of product-forms. \\

In this paper, we derive the stability condition of a nearest neighbour random walk using its underlying QBD structure. We do this by generalising the QBD drift conditions for random walks in the lattice \cite{Neuts} and  by illustrating the connection of these drift conditions to the Lyapunov drift conditions for random walks \cite{FIM}.\\

Using as a vehicle for illustration the sub-class of random walks in the quadrant with an invariant measure representable by a series of product-forms, we show how to obtain recursively the eigenvalues and the corresponding eigenvectors of the infinite matrix $\bm{R}$ appearing in the matrix geometric approach. Furthermore, we provide insights in the spectral properties of the infinite matrix $\bm{R}$ and discuss  in detail the properties of the resolvent operator of  $\bm{R}$. This work can be easily extended to cover a wider spectrum of random walks with meromorphic probability generating functions.\\

In this work, we set the foundations for the connection of three methodological approaches (the matrix geometric approach, the compensation approach, and  the boundary value problem method) and for gaining insight on the probabilistic interpretation of the terms appearing in the product-forms of the invariant measure.

\section*{Acknowledgments}
\noindent The work of S. Kapodistria is supported by an NWO Gravitation Project, NETWORKS, and a TKI Project, DAISY4OFFSHORE. The work of Z. Palmowski is partially supported by the National Science Centre under the grant
2015/17/B/ST1/01102.
The authors would like to thank I.J.B.F. Adan (Eindhoven University of Technology) and  O.J. Boxma (Eindhoven University of Technology) for their time and advice in the preparation of this work. Furthermore, the authors would like to thank  M. Miyazawa (Tokyo University of Science), P. Taylor (University of Melbourne), M. Telek (Technical University of Budapest), and G. Latouche (Universit\'e Libre de Bruxelles) for their valuable comments and the very interesting discussions. Finally, the authors would like to thank the anonymous reviewers for their careful reading of the paper and their insightful comments and suggestions.


\begin{thebibliography}{99}

\bibitem{MR1138205}
\textsc{Adan, I.J.B.F.} (1991).
\textit{A Compensation Approach for Queueing Problems}.
 PhD~ dissertation, Eindhoven University of Technology, The Netherlands.

\bibitem{MR1833660}
\textsc{Adan, I.J.B.F.}, \textsc{Boxma, O.J.} and \textsc{Resing  J.A.C.} (2001).
Queueing models with multiple waiting lines.
\textit{Queueing Systems}, {\bf 37}(1-3) 65--98.

\bibitem{Adan1}
\textsc{Adan, I.J.B.F.}, \textsc{Boxma, O.J.},  \textsc{Kapodistria, S.} and \textsc{Kulkarni, V.G.} (2016).
The shorter queue polling model.
\textit{Annals of Operations Research}, {\bf 241}(1-2) 167--200.

\bibitem{Adan2}
\textsc{Adan, I.J.B.F.}, \textsc{Kapodistria, S.} and \textsc{Van Leeuwaarden, J.S.H.} (2013).
Erlang arrivals joining the shorter queue.
\textit{Queueing Systems}, {\bf 74}(2-3) 273--302.

\bibitem{MR1241929}
\textsc{Adan, I.J.B.F.}, \textsc{Wessels,  J.} and  \textsc{Zijm,   W.H.M.} (1993).
 A compensation approach for two-dimensional {M}arkov processes.
 \textit{Advances in Applied Probability}, {\bf 25}(4) 783--817.



\bibitem{ChenPhD}
\textsc{Chen, Y.} (2016).
\textit{Random Walks in the Quarter-Plane: Invariant Measures and Performance Bounds}.
PhD~ dissertation, University of Twente, The Netherlands.


\bibitem{Chen}
\textsc{Chen, Y.}, \textsc{Boucherie, R.J.} and \textsc{Goseling, J.} (2016).
Invariant measures and error bounds for random walks in the quarter-plane based on sums of geometric terms.
\textit{Queueing Systems}, {\bf 84}(1-2) 21--48.

\bibitem{CohenJAP}
\textsc{Cohen, J.W.} (1994).
On a Class of Two-Dimensional Nearest-Neighbour Random Walks.
\textit{Journal of Applied Probability}, {\bf 31} 207--237 .

\bibitem{Cohen-asq}
\textsc{Cohen, J.W.} (1998).
Analysis of the asymmetrical shortest two-server queueing model.
 \textit{Journal of Applied Mathematics and Stochastic Analysis}, {\bf11}(2) 115--162.


\bibitem{Cohen_Boxma}
\textsc{Cohen, J.W.} and \textsc{Boxma, O.J.} (1983).
 \textit{Boundary Value Problems in Queueing System Analysis},
North-Holland, Amsterdam.

\bibitem{Due}
\textsc{Ertiningsih, D.}, \textsc{Katehakis, M.N.}, \textsc{Smit, L.C.}, and \textsc{Spieksma, F.M.} (2016). Level product form QSF processes and an analysis of queues with Coxian interarrival distribution.
\textit{Naval Research Logistics} DOI: 10.1002/nav.21680.


\bibitem{Fayolle_Iasnogorodski}
\textsc{Fayolle, G.} and  \textsc{Iasnogorodski, R.} (1979).
Two coupled processors: the reduction to a Riemann-Hilbert problem.
\textit{Zeitschrift f\"ur Wahrscheinlichkeitstheorie und Verwandte Gebiete}, {\bf 47}  325--351.


\bibitem{FIM}
\textsc{Fayolle, G.}, \textsc{Iasnogorodski, R.} and \textsc{Malyshev, V.} (1999).
\textit{Random Walks in the Quarter Plane},
Springer-Verlag, New York.



\bibitem{Haque}
\textsc{Haque, L.}, \textsc{Zhao, Y.Q.}, and \textsc{Liu, L.} (2005).
Sufficient conditions for a geometric tail in a QBD process with many countable levels and phases. \textit{Stochastic Models}, {\bf21}(1), 77--99.


\bibitem{QR}
\textsc{Hansen, A.C.} (2008).
On the approximation of spectra of linear operators on Hilbert spaces.
\textit{Journal of Functional Analysis}, {\bf 254}(8) 2092--2126.

\bibitem{QR2}
\textsc{Hansen, A.C.} (2010).
Infinite-dimensional numerical linear algebra: theory and applications.
In \textit{Proceedings of the Royal Society of London A: Mathematical, Physical and Engineering Sciences}, 466, pp. 3539--3559.


\bibitem{Hunter}
\textsc{Hunter J.} and \textsc{Nachtergaele B.} (2001).
\textit{Applied Analysis},
World Scientific Publishing Co., Singapore.

\bibitem{Gertsbakh}
\textsc{Gertsbakh, I.} (1984).
The shorter queue problem: A numerical study using the matrix-geometric solution.
\textit{European Journal of Operational Research}, {\bf 15}(3) 374--381.

\bibitem{Katehakis2}
\textsc{Katehakis, M.N.} and \textsc{Smit, L.C.} (2012).
A successive lumping procedure for a class of Markov chains.
\textit{Probability in the Engineering and Informational Sciences}, {\bf26}(04) 483--508.

\bibitem{Katehakis1}
\textsc{Katehakis, M.N.}, \textsc{Smit, L.C.}, and \textsc{Spieksma, F.M.} (2016).
A comparative analysis of the successive lumping and the lattice path counting algorithms.
\textit{Journal of Applied Probability}, {\bf53}(1) 106--120.


\bibitem{Kingman}
\textsc{Kingman, J.F.} (1961).
Two similar queues in parallel.
\textit{The Annals of Mathematical Statistics}, {\bf 32}(4) 1314--1323.

\bibitem{Taylor}
\textsc{Kroese, D.P.}, \textsc{Scheinhardt, W.R.W.}, and \textsc{Taylor, P.G.} (2004).
Spectral properties of the tandem Jackson network, seen as a quasi-birth-and-death process.
\textit{Annals of Applied Probability}, {\bf 14}(4) 2057--2089.


\bibitem{Latouche1}
\textsc{Latouche, G.}, \textsc{Nguyen, G.T.} and \textsc{Taylor, P.G.} (2011).
Queues with boundary assistance: the effects of truncation.
\textit{Queueing Systems}, {\bf 69}(2) 175--197.

\bibitem{Latouche}
\textsc{Latouche, G.}, and \textsc{Ramaswami, V.} (1999).
{\em Introduction to Matrix Analytic Methods in Stochastic Modeling}. Society for Industrial and Applied Mathematics (SIAM), Philadelphia.

\bibitem{Li}
\textsc{Li, H.}, \textsc{Miyazawa, M.}, and \textsc{Zhao, Y.Q.} (2007).
Geometric decay in a QBD process with countable background states with applications to a join-the-shortest-queue model.
\textit{Stochastic Models}, {\bf 23}(3), 413--438.

\bibitem{Lorch}
\textsc{Lorch, E.R.} (1962).
\textit{Spectral Theory}, Oxford University Press, New York

\bibitem{Miyazawa}
\textsc{Miyazawa, M.} (2009).
Tail decay rates in double QBD processes and related reflected random walks.
\textit{Mathematics of Operations Research}, {\bf 34}(3) 547--575.

\bibitem{24}
\textsc{Mohanty, S.} (1979).
\textit{Lattice Path Counting and Applications.}
Academic Press, New York, NY.

\bibitem{Motyer}
\textsc{Motyer, A.J.} and \textsc{Taylor, P.G.} (2006).
Decay rates for quasi-birth-and-death processes with countably many phases and tridiagonal block generators.
\textit{Advances in Applied Probability}, {\bf 38}(02), 522--544.


\bibitem{Neuts}
\textsc{Neuts, M.F.} (1981).
\textit{Matrix-geometric Solutions in Stochastic Models: An Algorithmic Approach},
The Johns Hopkins University Press, Baltimore.

\bibitem{Reed}
\textsc{Reed, M.} and \textsc{Simon, B.}(1980).
\textit{Methods of Modern Mathematical Physics: Functional Analysis}, vol. I., Academic Press, San Diego.

\bibitem{Ross}
\textsc{Ross, S. M.} (2007).
\textit{Introduction to Probability Models}, ninth edition, Elsevier, San Diego.

\bibitem{Shivakumar}
\textsc{Shivakumar, P.N.} and \textsc{Sivakumar, K.C.} (2009).
A review of infinite matrices and their applications.
\textit{Linear Algebra and its Applications}, {\bf 430}(4) 976--998.


\bibitem{Selen}
\textsc{Selen, J.}, \textsc{Adan, I.J.B.F}, \textsc{Kapodistria, S.} and \textsc{Van Leeuwaarden, J.S.H.}
Steady-state analysis of shortest expected delay routing.
\textit{Queueing Systems}, {\bf 84}(3-4) 309--54.

\bibitem{Takahashi}
\textsc{Takahashi, Y.}, \textsc{Fujimoto, K.}, and \textsc{Makimoto, N.} (2001).
Geometric decay of the steady-state probabilities in a quasi-birth-and-death process with a countable number of phases.
\textit{Stochastic Models}, {\bf 17}(1) 1--24.




\end{thebibliography}
\end{document}